\documentclass[12pt]{amsart}

\usepackage{amsmath}

\input xy \xyoption{all}

\DeclareMathOperator\C{\mathbb C}
\DeclareMathOperator\Z{\mathbb Z}

\DeclareMathOperator\Q{\mathbb Q}
\DeclareMathOperator\Hom{\mathrm Hom}
\DeclareMathOperator\Gr{\mathrm Gr}
\newcommand{\Res}{\operatornamewithlimits{Res}}

\DeclareMathOperator\Tp{\mathrm Tp}

\DeclareMathOperator\KTp{\mathrm KTp}
\DeclareMathOperator\G{\mathfrak G}
\DeclareMathOperator\A{\mathcal A}

\DeclareMathOperator\codim{codim}
\DeclareMathOperator\pt{\mathrm{pt}}

\DeclareMathOperator\Ltd{\Lie(T)^\lor}
\DeclareMathOperator\chik{\widetilde{\chi}}

\newcommand\GL{\mathrm{GL}}
\newcommand\calO{\mathcal{O}}
\newcommand\ep{\alpha}
\newcommand{\Wl}{\mathcal{W}}
\newcommand{\tr}{\mathrm{Tr}}

\newcommand{\Ce}{\C^a}

\newcommand{\Cf}{\C^b}
\newcommand{\Cb}{\C^b}
\newcommand{\glef}{\mathrm{GL}[a\rightarrow b]}
\newcommand{\gla}{\mathrm{GL}_a}
\newcommand{\glb}{\mathrm{GL}_b}

\def\gr{Gro\-then\-dieck }

\newtheorem{fact}{Fact}[section]
\newtheorem{lemma}[fact]{Lemma}
\newtheorem{assumption}[fact]{Assumption}
\newtheorem{theorem}[fact]{Theorem}
\newtheorem{definition}[fact]{Definition}
\newtheorem{example}[fact]{Example}
\newtheorem{rremark}[fact]{Remark}
\newenvironment{remark}{\begin{rremark} \rm}{\end{rremark}}
\newtheorem{proposition}[fact]{Proposition}

\newcommand{\Lie}{\mathrm{Lie}}

\author{Rich\'ard Rim\'anyi}
\address{Department of Mathematics, University of North Carolina at Chapel Hill, USA}
\email{rimanyi@email.unc.edu}

\author{Andr\'as Szenes}
\address{Section de math\'ematiques, Universit\'e de Gen\`eve, Switzerland}
\email{andras.szenes@unige.ch}

\title{Residues, Grothendieck polynomials and K-theoretic Thom polynomials}

\begin{document}

\begin{abstract}
  \gr polynomials were introduced by Lascoux and Sch\"utzenberger, and
  they play an important role in K-theoretic Schubert calculus. In
  this paper, we give a new definition of double stable \gr
  polynomials based on an iterated residue operation.

  We illustrate the power of our definition 
  by calculating the \gr expansion of K-theoretic Thom polynomials of
  $\A_2$ singularities. We present the expansion in two versions: one
  displays its expected stabilization property, while the other
  displays its expected finiteness property. 
\end{abstract}

\maketitle

\section{Introduction}
From the point of view of enumerative geometry, the most important
invariant of a subvariety $X$ in a smooth variety $M$ is its
cohomological fundamental class
$[X\subset M]\in H^{\codim(X\subset M)}(M)$, obtained from the
homology fundamental class by Poincar\'e duality.  A general strategy
to study this invariant is {\em degeneracy loci theory} (see
eg. \cite{fultonpragacz,bsz}), which reduces the problem of
calculating fundamental classes to calculating $G$-equivariant
fundamental classes
\[ [\eta \subset J] \in H_G^{\codim(\eta \subset
  J)}(J)=H_G^{\codim(\eta \subset J)}(\pt)\]
of $G$-invariant subvarieties $\eta$ of a $G$-representation $J$. 

We encounter this setup, for example, in modern Schubert calculus, where $J$ is
a representation vector space of a quiver and the fundamental class is
called a {\em quiver polynomial}, see e.g. \cite{KMS, buch, Rqr}. Another instance is  {\em global singularity theory},
where $J$ is the vector space of germs of maps acted upon by
reparametrizations, and the fundamental class is called the {\em Thom
  polynomial} \cite{Thom, RRTPP} of the
singularity.

In this paper we will be concerned with the notion of the {\em $G$-equivariant
K-theoretic fundamental class} $[\eta\subset J]\in K_G(J)=K_G(\pt)$ of
an invariant subvariety $\eta$ of a $G$-representation $J$.
It turns out that  there is some ambiguity in the definition of such
an object  (cf. Section \ref{sec:fundclass}), but , regardless, the cohomo\-lo\-gical
fundamental class may always be recovered from the
K-theoretic fundamental class via a limiting procedure.

In cohomological fundamental class theory, a natural basis is the
Schur basis, in part because the Schur polynomials are related to the
fundamental classes of Schubert varieties.  It is thus natural to
attempt to express K-theoretic fundamental classes in terms of {\em
  \gr polynomials} introduced in \cite{LS} which are similarly
related to the K-theoretic fundamental classes of (the structure
sheaves of) Schubert varieties.

There is a number of ``flavors'' of Schur and \gr polynomials, and, in
this article, we will focus on the so-called {\em double stable}
polynomials, which are best adapted to the bivariant problems we
study.

Finally, we would like to mention a central aspect of the theory: the
various {\em positivity} results, which state that in a number of
situations the coefficients of equivariant Poincar\'e duals in the
Schur basis are nonnegative. This is true, for example,  for the above mentioned Thom
polynomials of singularities \cite{pw}, and quiver polynomials \cite{KMS, BR}. In the K-theoretic
setup, it seems that the corresponding notion is expressions with {\em
  alternating signs}. An example of this phenomenon is Buch's result in \cite{buch}, which shows that
K-theoretic Dynkin quiver polynomials may be expressed in terms of
certain double stable \gr polynomials; moreover, the coefficients in
this expression are (conjecturally) alternating. \footnote{Using
  results of the present paper, Allman \cite{AllmanK}  showed
  stabilization properties of such expansions.}

In this paper, we present a residue calculus for double stable Grothendieck
polynomials, which makes proving various properties, in particular,
positivity of expansions, straightforward. 
Our formulas allow us to begin the study of K-theoretic Thom polynomials. 
Below, we give a quick introduction to these two subjects and explain
the main results of the paper.

\subsection{\gr polynomials} \label{sec:introgr}

In Section \ref{sec:combgr}, we recall the original definition of
double stable \gr polynomials \cite{FK1, FK2}. This involves first
introducing {\em ordinary \gr polynomials} $\G_w$, indexed by
permutations, and defined by a recursion involving divided
differences. Geometrically, the polynomials $\G_w$ represent
torus-equivariant K-theoretic fundamental classes of Schubert
varieties in full flag varieties. Next, {\em double stable \gr
  polynomials} $G_\lambda(\alpha;\beta)$ parametrized by partitions
are defined by a limiting procedure from ordinary \gr polynomials,
and, finally, applying to these latter polynomials a set of certain
{\em straightening laws}, one defines double stable \gr polynomials
$G_I(\alpha;\beta)$ parametrized by arbitrary integer sequences.
Another approach to double stable \gr polynomials parametrized by
partitions uses the combinatorics of set-valued tableaux
\cite{B:combgr}.

In \S \ref{sec:defg}, we propose a new formula for the most general
integer-sequence parametrized double stable \gr polynomials:
\begin{multline}\label{eqn:introG}
G_I(\alpha_1,\ldots,\alpha_k;\beta_1,\ldots,\beta_l)=\\
\Res_{z_1=0,\infty}\ldots \Res_{z_r=0,\infty}  \left( \prod_{j=1}^r (1-z_j)^{I_j-j} 
\prod_{i>j} \left( 1-\frac{z_i}{z_j} \right)
\prod_{j=1}^r
\frac{ \prod_{i=1}^l (1-z_j\beta_i) }{ \prod_{i=1}^k (1-z_j\alpha_i) (1-z_j)^{l-k}}
\prod_{j=1}^r \frac{dz_j}{z_j}
\right).
\end{multline}
This formula is analogous to the useful residue formula
\begin{multline}\label{eqn:schurdef}
s_I(\bar{\alpha}_1,\ldots,\bar{\alpha}_k;\bar{\beta}_1,\ldots,\bar{\beta}_l)=\\(-1)^r
\Res_{z_1=\infty}\ldots\Res_{z_r=\infty}
\left(
\prod_{j=1}^{r} z_j^{I_j} 
\prod_{j>i } \left(1-\frac{z_i}{z_j}\right)
\prod_{j=1}^{r}
\frac{  \prod_{i=1}^l (1+\bar{\beta}_i/z_j) }{ \prod_{i=1}^k (1+\bar{\alpha}_i/z_j)}
\cdot
\prod_{j=1}^r \frac{dz_j}{z_j}
\right)
  \end{multline}
  for the double stable Schur polynomials (see e.g. \cite[Lemma
  6.1]{csm}). Note that in the case of Schur polynomials, the residues
  are taken only at infinity, while for \gr polynomials, one takes the
  sum of the residues at 0 and infinity.

  In addition to its simplicity and efficient computability, our
  formula (\ref{eqn:introG}) has the perfect form for computing the
  expansions of K-theoretic Thom polynomials.

\subsection{K-theoretic Thom polynomials of
  singularities} \label{sec:introktp}
The general reference for singularities of maps is \cite{AVGL}. 
For a positive integer $N$, denote by  $R^N(\Ce)$ the algebra of
$N$-jets of functions on $\Ce$ at 0; this is the ring of polynomials
in $a$ variables modulo monomials of degree at
least $N+1$.  Let $J^N(\Ce,\Cf)$ be the space of $N$-jets  of maps
$(\Ce,0)\to (\Cf,0)$ vanishing at 0. An element of $J^N(\Ce,\Cf)$ is given by a $b$-tuple of jets from the maximal ideal of $R^N(\Ce)$.
A {\em singularity} $\eta$ is an algebraic subvariety of $J^N(\Ce,\Cf)$
invariant under the group of formal holomorphic 
reparametrizations of $(\Ce,0)$ and $(\Cf,0)$ (cf. e.g. \cite{bsz}
\S3).

  An important set of examples of singularities, called {\em contact
    singularities}, is obtained as follows. A key reparametrization
  invariant of $N$-jets of function is the {\em local algebra},
  defined for
  $h=(h_1(x_1,\ldots,x_a),\ldots,h_b(x_1,\ldots,x_a))\in J^N(\Ce,\Cf)$
  as the ideal quotient $R^N(\Ce)/(h_1,\ldots,h_b)$. Then for a fixed
  finite-dimensional local commutative algebra $Q$ and nonnegative
  integers $a\leq b$, we can define the singularity $\eta_Q^{a\to b}$
  as the Zariski closure of the set
\[
\{g\in J^N(\Ce, \Cf):\text{ the local algebra of $g$ is isomorphic to } Q\}.
\]
(We will omit the dimensions $a$ and $b$ from the notation when this
causes no confusion.)

Denote the group of linear reparametrizations
$\gla(\C)\times \glb(\C)$ by $\glef$, and observe that the space
$J^N(\Ce,\Cf)$ is equivariantly contractible, hence we have the
identification with the symmetric polynomials:
\begin{align*}
 H^*_{\glef}(J^N(\Ce,\Cf))=H^*_{\glef}(\pt)=\Z[\bar{\ep}_1,\ldots,\bar{\ep}_a,\bar{\beta}_1,\ldots,\bar{\beta}_b]^{S_a\times S_b}, \\
 K_{\glef}(J^N(\Ce,\Cf))=K_{\glef}(\pt)=\Z[{\ep}_1^{\pm1},\ldots,{\ep}_a^{\pm1},{\beta}_1^{\pm 1},\ldots,{\beta}_b^{\pm 1}]^{S_a\times S_b},
\end{align*}
where $\bar{\ep}_i$  ($\bar{\beta}_j$),  and $\ep_i$,
($\beta_j$)    are the cohomological and  K-theoretic Chern roots of the
standard representation of $\gla(\C)$ ($\glb(\C)$), and $S_m$ is the
permutation group on $m$ elements .

In \S\ref{sec:fundclass}, we recall the definition of the equivariant
Poincar\'e dual class $[X]$ of an invariant algebraic subvariety
$X\subset V$ in a vector space  acted upon by a Lie group. Using this
notion, we define the  {\em  Thom polynomial} of the singularity $\eta$ as the equivariant
Poincar\'e dual 
\[   \Tp_\eta^{a\to b}=[\eta]\in H^*_{\glef}(J^N(\Ce,\Cf)).
\]
 The analogous K-theoretic notion
\[ \KTp^{a\to b}_\eta=[\eta]^K\in K_{\glef}(J^N(\Ce,\Cf))
\]
 is, in fact, somewhat problematic, and we will
discuss its definition in detail in  Section \ref{sec:fundclass} as well.  

To simplify our notation, we will denote the Thom polynomial of the
contact singularity $\eta_Q$ as $\Tp_Q$ (and $\KTp_Q$) when this
causes no confusion. Consider the example of $Q=\A_2=\C[x]/(x^3)$. We
will write formulas for $\Tp_{\A_2}$ in terms of Schur
functions
$s_\lambda=s_\lambda(\bar{\ep_1},\ldots,\bar{\ep_a},\bar{\beta_1},\ldots,\bar{\beta_b})$
defined in (\ref{eqn:schurdef}), or equivalently, by the more
standard definition $s_\lambda=\det(c_{\lambda(i)+j-i})$ with
\[
1+c_1t+c_2t^2+\ldots=\frac{\prod_{i=1}^b(1+\bar{\beta}_it)}{\prod_{i=1}^a(1+\bar{\ep}_i)}.  
\]
The general formula due to Ronga \cite{ronga} is as follows:
\begin{equation} \label{eqn:A2gen}
\Tp_{\A_2}^{a \to a+l}=\sum_{i=0}^{l+1} 2^{i}s_{l+1+i,l+1-i}.
\end{equation}
Here are the first few cases:
\[
\Tp_{\A_2}^{a\to a}=s_{1,1}+2s_{2,0},\qquad
\Tp_{\A_2}^{a\to a+1}=s_{2,2}+2s_{3,1}+4s_{4,0},\qquad
\Tp_{\A_2}^{a\to a+2}=s_{3,3}+2s_{4,2}+4s_{5,1}+8s_{6,0}.\]

Formula (\ref{eqn:A2gen}) illustrates
three key features of cohomological Thom polynomials of contact
singularities:
\begin{itemize}
\item (stability) The Thom polynomial $\Tp_{\A_2}^{a\to b}$ only depends on
  the {\em relative dimension} $b-a$ (denoted by $l$), not on $a$ and $b$
  individually.
\item ($l$-stability) We obtain $\Tp_{\A_2}^{a\to a+l}$ from
  $\Tp_{\A_2}^{a\to a+l+1}$ by replacing each Schur polynomial $s_{a,b}$ by
  $s_{a-1,b-1}$ (note that $s_{a,-1}=0$). The general
  statement of this property for arbitrary $Q$ may be found in \cite[Theorems 2.1, 4.1]{structure}.
\item (positivity) The coefficients of Schur expansions of Thom
  polynomials of contact singularities are non-negative \cite{pw}.
\end{itemize}

\smallskip

In \S\ref{sec:A2} we calculate the K-theoretic Thom
polynomials $\KTp_{\A_2}^{a\to b}$ for all $a\leq b$, and in 
\S\ref{sec:higher} we comment on the case of 
higher singularities. 

In our calculations, we observe a feature new to K-theory: our K-theoretic Thom polynomials have
two different types of expansions. 

The first type begins with a formal
infinite sum of \gr polynomials indexed by integer sequences; this
infinite sum {\em has the $l$-stability property} analogous to the
$l$-stability of cohomological Thom polynomials, see Remark
\ref{rem:lstab}. Partially summing the series, one obtains a reduced
series, whose all but finitely many terms vanish. 

The second expansion, which we will call {\em minimal} 
(Theorem \ref{thm:expansion2}), expresses KTp as a finite sum of \gr
polynomials indexed by partitions. This expression is uniquely defined
but it  is not $l$-stable.

Let us give a visual explanation of the relation between the two \gr
expansions of $\KTp_{\A_2}$. Consider the rational function
\[f(x_1,x_2)=\left.\frac{1}{1-z_2/z_1^2}\right|_{z_1=1-x_1,z_2=1-x_2}=\frac{1-2x_1+x_1^2}{x_2-2x_1+x_1^2}.\]
The coefficients of its $|x_1|<|x_2|$ Laurent expansion are naturally arranged in the infinite grid
\[
\xymatrix @R=1pc {
 & 1 & x_1 & x_1^2 & x_1^3 & x_1^4 & x_1^5 & x_1^6 & x_1^7 & x_1^8 \\
x_2^{-1} & 1 & -2 & 1 \\
x_2^{-2} &  & 2  & -5 & 4 & -1\\
x_2^{-3} &  &    &  4 & -12 & 13 & -6 & 1 & & \\
x_2^{-4} &  & & & 8 \ar@/_/[u] & -28 \ar@/_/[u] & 38 \ar@/_/[u]& -25  \ar@/_/[u]& 8  \ar@/_/[u] & -1 \ar@/_/[u]  \\
x_2^{-5} &  & & &   & 16 \ar@/_1.3pc/[uu] & -64\ar@/_1.2pc/[uu] & 104\ar@/_1.2pc/[uu] & \ldots & \ldots & \\
x_2^{-6} &  & & & & & 32\ar@/_2pc/[uuu] & \ldots & \ldots & \ldots. &\\
\save "1,2"."1,10"*[F]\frm{} \restore
\save "2,1"."7,1"*[F]\frm{} \restore
\save "4,4"."4,4"*[F]\frm{} \restore
\save "4,5"."5,5"*[F]\frm{} \restore
\save "4,6"."6,6"*[F]\frm{} \restore
\save "4,7"."7,7"*[F]\frm{} \restore
\save "4,8"."7,8"*[F]\frm{} \restore
\save "5,9"."7,9"*[F]\frm{} \restore
\save "5,10"."7,10"*[F]\frm{} \restore
}
\]
In the {\em formal stable version} of \gr expansion of
$\KTp_{\A_2}^{a\to a+l}$ (Theorem \ref{thm:expansion}), these numbers
are exactly the coefficients of the corresponding \gr polynomials, for
{\em any $l$}, with an appropriate shift. To obtain a finite
expression, we sum these \gr polynomials first
in the vertical direction, and, as will we show, all but finitely many
of these partal sums will vanish, giving a correct finite expression
for $\KTp_{\A_2}$.

To obtain the  {\em minimal version} of our formula,
Theorem \ref{thm:expansion2}, the coefficients of $\KTp_{\A_2,a,a+l}$ for different $l$'s are obtained by 
different procedures from this grid of integers. For example, for $l=1$ we ``sweep up'' all numbers
from {\em below} the third row {\em to} the third row. That is,
replace the $(3,k)$ entry with the sum of entries $(r,k)$ for
$r\geq 3$ and then delete the rows from the 4th one down. This
sweeping is illustrated by the framed entries in the
picture. In the resulting table we get the numbers (reading along the
diagonals) $1,2,4$;  $-2,-5,-12+8=-4$; $1,4,13-28+16=1$; $-1$, and then {\em infinitely
many 0's}. These are exactly the coefficients in the {\em minimal} \gr
expansion of $\KTp_{\A_2}^{a\to a+1}$,  cf. (\ref{eqn:KTPsmall}).  To get $\KTp_{\A_2}^{a \to a+2}$ we
need to ``sweep'' the same table below the 4th row, for $l=3$ we sweep
from the 5th row, etc. The exact statement of this sweeping procedure
is given in Theorem \ref{thm:expansion2}.

As a result, we obtain the following minimal expansions:
\begin{align} \label{eqn:KTPsmall}
\KTp_{\A_2}^{a\to a}=&\big(G_{1,1}+2G_2\big)-\big(2G_{2,1}+G_3\big) + G_{3,1} \\
\notag \KTp_{\A_2}^{a\to a+1}=&\big(G_{2,2}+2G_{3,1}+4G_4\big)-\big(2G_{3,2}+5G_{4,1}+4G_5\big)+\big(G_{4,2}+4G_{5,1}+G_6\big) -G_{6,1}\\
\notag 
\begin{split} \KTp_{\A_2}^{a\to a+2}=&\big(G_{3,3}+2G_{4,2}+4G_{5,1}+8G_6\big)-\big(2G_{4,3}+5G_{5,2}+12G_{6,1}+12G_7\big) \\
& +\big(G_{5,3}+4G_{6,2}+13G_{7,1}+6G_8\big) -\big(G_{7,2}+6G_{8,1}+G_{9}\big)+G_{9,1}. 
\end{split}
\end{align}

\smallskip

It is remarkable that the third key feature, the positivity of
cohomological Thom polynomials extends to a rule of alternating signs
for {\em both} of our expansions. This result will be proved in
\S \ref{sec:pos}.

\medskip

\noindent {\bf Acknowledgement.}
The first author was supported by the Simons Foundation grant 523882. He is also grateful for the hospitality and financial
support of University of Geneva and CIB during his stay there while
parts of this research was done. The second author is partially
supported by FNS grants 156645, 159581 and 175799.

\section{Combinatorial definition of \gr polynomials}\label{sec:combgr}
In this section we will review the traditional definition of various versions of \gr polynomials. We
follow the references \cite{LS, FK1, FK2, B:Gr,B:combgr,
  buch}. 
Our goal in Sections 2-4 is to replace these traditional definitions with the residue
description of Definition \ref{def:g}. The reader not interested in
the traditional definitions can take Definition \ref{def:g} to be the
definition of double stable \gr polynomials and jump to Section
\ref{sec:fundclass}.

\medskip

We will use standard notations of algebraic combinatorics.
A permutation $w\in S_n$ will be represented by the sequence
$[w(1),w(2),\ldots,w(n)]$. The length of a permutation $\ell(w)$
is the cardinality of the set
$\{i<j: w(i)>w(j)\}$. We will identify $S_n$ with its image under the
natural embedding $S_n=\{w\in S_{n+1}|\; w(n+1)=n+1\}$.

\subsection{Double \gr polynomials} \label{sec:rew}
Double \gr polynomials  (in variables $x_i$, $y_j$) were introduced by Lascoux
and Schut\-zen\-ber\-ger \cite{LS}. In the present paper, following e.g.
\cite{B:Gr}, we perform the
rational substitutions $x_i=1-1/\alpha_i$ and $y_i=1-\beta_i$ in those
polynomials, and denote the resulting rational functions by
$\G_w(\alpha,\beta)$. To keep the terminology simple, we will continue
calling these functions ``\gr polynomials''.

The functions $\G_w(\alpha,\beta)$ are defined by the following  recursion:
\begin{itemize}
\item{} For the longest permutation $w_0=[n,n-1,\ldots,1] \in S_n$, let
  $$\G_{w_0}=\prod_{i+j\leq n} \left(1-\frac{\beta_i}{\alpha_j}\right).$$
\item{} Let $s_i$ be the $i$th elementary transposition. If $\ell(w s_i)=\ell(w)+1$ then
  $$\G_w=\pi_i(\G_{w s_i}),$$
  where the isobaric divided difference operator $\pi_i$ is defined by
  
   \begin{align*}
     \pi_i(f) & =  \frac{\alpha_i f(\ldots, \alpha_i,\alpha_{i+1},
                \ldots) - \alpha_{i+1}f( \ldots, \alpha_{i+1},\alpha_i,
                \ldots)}{\alpha_i-\alpha_{i+1}} \\
              & =  \frac{f(\ldots, \alpha_i,\alpha_{i+1}, \ldots)}{1-\alpha_{i+1}/\alpha_i} + \frac{f( \ldots ,\alpha_{i+1},\alpha_i, \ldots)}{1-\alpha_i/\alpha_{i+1}} .
   \end{align*}
\end{itemize}

\noindent For example, here is the list of double \gr polynomials for all $w\in S_3$
$$\G_{321}=\left( 1 -\frac{\beta_1}{\alpha_1}\right) \left( 1 -\frac{\beta_2}{\alpha_1}\right) \left( 1 -\frac{\beta_1}{\alpha_2}\right)  \qquad
\G_{231}=\left( 1 -\frac{\beta_1}{\alpha_1}\right) \left( 1 -\frac{\beta_1}{\alpha_2}\right)$$
$$\G_{312}=\left( 1 -\frac{\beta_1}{\alpha_1}\right) \left( 1 -\frac{\beta_2}{\alpha_1}\right) \qquad
\G_{213}=1 -\frac{\beta_1}{\alpha_1} \qquad \G_{132}= 1 -\frac{\beta_1\beta_2}{\alpha_1\alpha_2} \qquad \G_{123}=1.$$

\subsection{Stable versions} \label{sec:stable}

For a permutation $w\in S_n$ let $1^m \times w \in S_{m+n}$ be the permutation that is the identity on $\{1,\ldots,m\}$ and maps $j \mapsto w(j-m)+m$ for $j>m$. The double stable \gr polynomial $G_w(\alpha,\beta)$ is defined to be
\begin{equation}
  \label{eq:glim}
  G_w = \lim_{m\to \infty} \G_{1^m \times w}.
\end{equation}
For example,
$G_{21}=1-\frac{\beta_1\beta_2\beta_3\cdots}{\alpha_1\alpha_2\alpha_3\cdots}$. The
precise  definition of this limit may be found in
\cite{B:Gr}: roughly, rewritten in the $x$ and $y$ variables mentioned above, each
coefficient of $\G_{1^m \times w}$ stabilizes with $m$, and hence the
limit is defined as a formal power series in $x_i, y_j$ with the stabilized
coefficients.

\subsection{Truncated versions}\label{sec:trunc}
One usually considers specializations of double stable \gr polynomials of the type
\begin{equation} \label{eq:trunc}
 G_w^{k,l}(\alpha_1,\ldots,\alpha_k; \beta_1,\ldots,\beta_l)=
 G_w(\alpha_1,\ldots,\alpha_k,1,1,\ldots ; \beta_1,\ldots,\beta_l,1,1,\ldots).
\end{equation}
In fact, $G_w^{k,l}$ may be obtained by substituting
$\alpha_i=1, i>k$, $\beta_i=1, i>l$ in $\G_{1^m\times w}$ for
$m\gg k,l$. This way the truncated versions (\ref{eq:trunc}) may be
calculated without the $\lim_{m\to \infty}$ of \eqref{eq:glim}.

Below, we will drop the superscripts $k,l$ whenever they may be
determined from the number of $\alpha$ and $\beta$ variables.

In the case $l=0$,  we will simply write $G_w(\alpha_1,\ldots,\alpha_k)$.

\subsection{Stable \gr polynomials parametrized by partitions.} As
usual, a weakly decreasing sequence of nonnegative integers
$\lambda=(\lambda_1,\ldots,\lambda_r)$ will be called a {\em
  partition}. We will identify two partitions if they differ by a
sequence of $0$'s, and we define $L(\lambda)$, the {\em length} of a partition $\lambda$
to be the largest $i$ for which $\lambda_i>0$. The {\em Grassmannian
permutation associated to a partition $\lambda$ with descent in position $p$} is
the permutation 
\[
w_\lambda(i) =
\begin{cases}
  w_\lambda(i)=i+\lambda_{p+1-i} \text{ for }i\leq p,\text{ and }\\
w_\lambda(i)<w_\lambda(i+1) \text{ unless }i=p.
\end{cases}
 \]
 Note that necessarily $p\ge L(\lambda)$.

We define the {\em double stable \gr polynomial $G_\lambda$ of the partition $\lambda$} as
$G_{w_\lambda}(\alpha;\beta)$. It is easy to show that this definition does not depend on the choice of $p$ above.

\subsection{Stable \gr polynomials parametrized by integer sequences.} The notion $G_\lambda$ (with $\lambda$ a partition) is extended to $G_I$ where $I \in \Z^r$ is any finite integer sequence---by repeated applications of the {\em straightening laws}
\begin{eqnarray}\label{eqn:straightening1}
G_{I,p,q,J} & = & \sum_{k=p+1}^q G_{I,q,k,J} - \sum_{k=p+1}^{q-1} G_{I,q-1,k,J} \qquad \text{if}\ p<q,
\\
G_{I,p}& = & G_{I,0}=G_I \qquad \text{if}\ p<0. \label{eqn:straightening2}
\end{eqnarray}

\section{Properties of \gr polynomials}
We will need the following three properties of \gr polynomials.

\begin{proposition} \cite{FK2}, \cite[(2)]{buch} \label{prop:G super}
  The polynomial
  $G_w(\alpha_1,\ldots,\alpha_k;\beta_1,\ldots,\beta_l)$ is
  $S_k \times S_l$-supersymmetric, i.e. it is symmetric in the
  $\alpha_i$ and the $\beta_j$ variables separately, and satisfies
$$G_w(\alpha_1,\ldots,\alpha_{k-1},t; \beta_1,\ldots,\beta_{l-1},t)=G_w(\alpha_1,\ldots,\alpha_{k-1};
\beta_1,\ldots, \beta_{l-1}).$$
In particular, the left hand side of this equality does not depend on $t$.
\end{proposition}

The next statement is an easy application of the Fomin-Kirillov formulas
\cite{FK1}, and also follows directly from the set-valued tableau description in \cite{B:Gr}.

\begin{proposition}  \label{prop:G vanish}
Let $\lambda=(\lambda_1,\ldots,\lambda_r)$ be a partition with
$\lambda_r>0$ and  let $0<k<r$. Then $G_\lambda(\alpha_1,\ldots,\alpha_k)=0$.
\end{proposition}

\begin{proposition} \label{prop:many_vars}
Let $\lambda=(\lambda_1,\ldots,\lambda_r)$ be a partition with $\lambda_r\geq 0$. We have
\[
G_\lambda(\alpha_1,\ldots,\alpha_r)=\sum_{\sigma\in S_r}
\frac{  \prod_{i=1}^r \left( 1 - {1/\alpha_{\sigma(i)}}\right)^{\lambda_i+r-i} }
{\prod_{i>j} \left( 1- \alpha_{\sigma(i)} /\alpha_{\sigma(j)} \right)\ \ \ \ \ \ }.
\]
\end{proposition}

\begin{proof}
Consider the permutation
\[
\bar{w}_\lambda=
\lambda_1+r, \lambda_2+r-1, \ldots, \lambda_r+1, i_1, \ldots, i_s
\]
where $i_j<i_{j+1}$ for all $j$, and $s$ is sufficiently large to make
this a permutation.  The permutation $\bar{w}_\lambda$ is a so-called
dominant permutation. 
For dominant permutations the recursive definition
of Section \ref{sec:rew} can be solved explicitly (\cite{LS}, or see
the diagrammatic description in \cite{FK1}), and we obtain
\[
\G_{\bar{w}_\lambda}(\alpha_1,\ldots,\alpha_r)=\prod_{i=1}^r \left( 1 - \frac{1}{\alpha_{i}}\right)^{\lambda_i+r-i}.
\]
Observe that $\bar{w}_\lambda \cdot w_0 = w_\lambda$, where $w_0$ is the longest permutation of $1,\ldots,r$. Hence
\begin{equation}\label{lala}
\G_\lambda(\alpha_1,\ldots,\alpha_r)=\G_{w_\lambda}(\alpha_1,\ldots,\alpha_r)=\pi_{w_0(r)} \left(
\prod_{i=1}^r \left( 1 - \frac{1}{\alpha_{i}}\right)^{\lambda_i+r-i} \right),
\end{equation}
where 
\[\pi_{w_0(r)}(f)=(\pi_1\pi_2\ldots \pi_{r-1})(\pi_1\pi_2\ldots\pi_{r-2})\ldots(\pi_1)(f)=\sum_{\sigma\in S_r}\sigma\left(\frac{f}{\prod_{i>j}(1-\alpha_i/\alpha_j)}\right).
\]
The right-hand side of (\ref{lala}) is equal to the right-hand side of
the displayed formula in the Proposition.  If the number of $\alpha$
variables is at least the length of the partition, then
$\G_\lambda(\alpha)=G_\lambda(\alpha)$, which concludes our proof.
\end{proof}

Note that Proposition \ref{prop:many_vars} may be used whenever the number of $\alpha$ variables is larger than the length of the partition, because we can append 0's to the end of $\lambda$ to make the condition satisfied.

\section{\gr polynomials in residue form}\label{sec:ResGr}

In this section we introduce a residue calculus for \gr polynomials
and show how this new formalism helps to understand some of their
properties.

Let $z$ be a complex variable, and introduce the notation
\[\Res_{z=0,\infty} f(z)\,dz= \Res_{z=0} f(z)\,dz + \Res_{z=\infty}
f(z)\,dz.\]
The following property of $\Res_{z=0,\infty}$ is straightforward.
\begin{lemma}\label{lem:residue0}
Let $0\leq a\leq s-r-2$ and let
\[
f(z)=z^a\cdot \frac{\prod_{i=1}^r (z-x_i)}{\prod_{i=1}^s (z-y_i)}
\]
for non-zero complex numbers $x_i, y_i$. Then $\Res_{z=0,\infty} f(z)\,dz =0$. \qed
\end{lemma}

\subsection{Residue form of double stable \gr polynomials} \label{sec:defg}

Let $z_1,\ldots,z_r$ be complex variables. For nonnegative integers
$k,l$, define the differential form
\begin{equation}
  \label{eqn:defM}
M_{k,l}(z_1,\ldots,z_r)=
\prod_{j=1}^r
\frac{ \prod_{i=1}^l (1-z_j\beta_i) }{ \prod_{i=1}^k (1-z_j\alpha_i) (1-z_j)^{l-k}}
\cdot \prod_{j=1}^r \frac{dz_j}{z_j}.
\end{equation}
When it causes no confusion, we will omit the indices $k$ an $l$, and
denote the vector $(z_1,\dots z_r)$ by $z$: thus we will write $M(z)$
for $M_{k,l}(z_1,\ldots,z_r)$.
\begin{definition} \label{def:g} For an integer sequence $I\in \Z^r$,  define the
  $g$-polynomial as
\begin{equation} \label{eqn:defg}
g_I(\alpha_1,\ldots,\alpha_k; \beta_1,\ldots,\beta_l)=\Res_{z_1=0,\infty}\ldots \Res_{z_r=0,\infty} \left( \prod_{j=1}^r (1-z_j)^{I_j-j} \prod_{i>j} \left( 1-\frac{z_i}{z_j} \right)M_{k,l}(z_1,\ldots,z_r)
\right).
\end{equation}
\end{definition}

\begin{remark}
  \label{rem:normalcross}
In general, iterated residue formulas are sensitive to the order in
which one takes the residues $\Res_{z_i}$---see for example
\cite{bsz,kaza:gysin,kaza:noass,ts}---due to factors of the type
$z_i-z_j$ in the denominator. However, the denominators in
(\ref{eqn:defg}) are linear factors each depending on a single variable,
and hence the order in this case does not matter.
\end{remark}
The following is evident from Definition \ref{def:g}.
\begin{lemma} \label{lem:g_super}
  We have
  \begin{equation}
    \label{eqn:gred}
 g_I(\alpha_1,\ldots,\alpha_k,1;
  \beta_1,\ldots,\beta_l)=g_I(\alpha_1,\ldots,\alpha_k;
  \beta_1,\ldots,\beta_l,1)=g_I(\alpha_1,\ldots,\alpha_k;
  \beta_1,\ldots,\beta_l).
  \end{equation}
The function
$g_\lambda(\alpha_1,\ldots,\alpha_k;\beta_1,\ldots,\beta_l)$ is
supersymmetric: it is symmetric in the $\alpha_i$ and the $\beta_j$
variables separately, and we have
$$g_\lambda(\alpha_1,\ldots,\alpha_{k-1},t; \beta_1,\ldots,\beta_{l-1},t)=g_\lambda(\alpha_1,\ldots,\alpha_{k-1};
\beta_1,\ldots, \beta_{l-1}).$$
In particular, the left hand side does not depend on $t$. \qed
\end{lemma}

\begin{theorem} \label{thm:G=g}
For any integer sequence $I$, and nonnegative integers $k,l$, we have
\[
G_I(\alpha_1,\ldots,\alpha_k;\beta_1,\ldots,\beta_l)=g_I(\alpha_1,\ldots,\alpha_k;\beta_1,\ldots,\beta_l).
\]
\end{theorem}

First we prove two lemmas.

\begin{lemma}\label{lem:replace}
Let $I$ and $J$ be integer sequences. Then we have
\begin{equation}\label{eqn:grepl}
g_{I,p,q,J}=\sum_{k=p+1}^q g_{I,q,k,J} - \sum_{k=p+1}^{q-1} g_{I,q-1,k,J} \qquad \text{if}\ p<q,
\end{equation}
and
\begin{equation}\label{eqn:gneg}
g_{I,p}=g_{I} \qquad \text{if}\ p\leq 0.
\end{equation}
\end{lemma}

\begin{proof}
For  simplicity of notation,  we assume that $I=J=\emptyset$. The
general case is treated similarly.
For $p<q$ consider
\[ g_{p,q}-g_{q-1,q} = \Res_{z_1=0,\infty} \Res_{z_2=0,\infty} \left(
  (1-z_1)^{p-1}(1-z_2)^{q-2} - (1-z_1)^{q-2}(1-z_2)^{q-2} \right)
\left(1-\frac{z_2}{z_1}\right) \cdot M(z_1,z_2).
\]
 Applying the identities
$$\left( 1-\frac{z_2}{z_1}\right)  = -\frac{z_2}{z_1}\left( 1-\frac{z_1}{z_2}\right)$$
and
\[
\left( (1-z_2)^{q-2}(1-z_1)^{p-1} - (1-z_2)^{q-2}(1-z_1)^{q-2} \right) \left(-\frac{z_2}{z_1}\right) =
\hskip 6 true cm\]
\[
\ \hskip 5 true cm
\sum_{k=p+1}^{q-1} (1-z_2)^{q-1}(1-z_1)^{k-2} - \sum_{k=p+1}^{q-1} (1-z_2)^{q-2}(1-z_1)^{k-2},
\]
we obtain that $g_{p,q}-g_{q-1,q}$ equals
$$\Res_{z_1=0,\infty} \Res_{z_2=0,\infty}
\left( \sum_{k=p+1}^{q-1} (1-z_2)^{q-1}(1-z_1)^{k-2} - \sum_{k=p+1}^{q-1} (1-z_2)^{q-2}(1-z_1)^{k-2} \right) \left(1-\frac{z_1}{z_2}\right) \cdot M(z_1,z_2).
$$
Using the definition of $g$ with the role of $z_1$ and $z_2$ switched,
we obtain
$$g_{p,q}-g_{q-1,q}=\sum_{k=p+1}^{q-1} g_{q,k} - \sum_{k=p+1}^{q-1} g_{q-1,k}.$$
This is equivalent to (\ref{eqn:grepl}) up to the easy equality
$$g_{q-1,q}=g_{q,q},$$
whose proof we leave to the reader.

Formula (\ref{eqn:gneg}) immediately follows from the fact that, for
$p\leq 0$,
$$ \Res_{z_r=0}(1-z_r)^{p-r} \prod_{i=1}^{r-1} \left(1-\frac{z_r}{z_i}\right) M_{k,l}(z_r)=1,$$
while the residue of this expression at $z_r=\infty$ vanishes.
\end{proof}

\begin{lemma} \label{lem:g_vanish} Let
  $\lambda=(\lambda_1,\ldots,\lambda_r)$ be a partition with
  $\lambda_r>0$. Then for $k<r$, we have $g_\lambda(\alpha_1,\ldots,\alpha_k)=0$.
\end{lemma}

\begin{proof} First note that, according to 
Lemma \ref{lem:g_super}, the case $k=r-1$  implies the case $k<r$.

Now assume $k=r-1$, and introduce the temporary notation $\omega$
for the differential form in~\eqref{eqn:defg}. Assume that the values
of the $\alpha_i$s are all different.
 
We calculate the first residue $\Res_{z_r=0,\infty}\omega$, taking
into account Remark \ref{rem:normalcross} and applying the 1-variable
Residue Theorem.  The exponent $I_r-r+k-l$ of the factor
$(1-z_r\beta_i)$ is nonnegative, since $I_r=\lambda_r>0$, $k=r-1$, and
$l=0$, and hence there is no pole at $z_r=1$. The remaining poles are
thus the points $z_r=1/\alpha_i$, $i=1,\ldots,k$, and each of these
poles is simple. The residue at the simple pole $z_r=1/\alpha_i$, up
to a factor of $-\alpha_i$ is obtained by omitting the factor
$(1-\alpha_iz_r)$ in the denominator, and then substituting into the
remainder $z_r=1/\alpha_i$. Continuing the application of residues in
\eqref{eqn:defg}, we obtain a sum over all choices of indices
$1\le i_j\le k$, $j=1,\dots r$, of terms of the following form
\[   \prod_{j=1}^r (1-\alpha_{i_j})^{\epsilon}
  \prod_{m>j}\left(1-\frac{\alpha_{i_m}}{\alpha_{i_j}}\right)
\widetilde{M},
\]
where $\epsilon\ge0$ and $\widetilde{M}$ is some rational expression
in the $\alpha$'s. The
relevant factor in the product is the second one, which vanishes as
long as $i_m=i_j$ for some $1\le j<m\le r$. As $k<r$, this is
certainly the case, and this completes the proof.
\end{proof}

Now we are ready to prove Theorem \ref{thm:G=g}.

\begin{proof}
  Since both $g$ and $G$ are supersymmetric (Proposition \ref{prop:G
    super} and Lemma \ref{lem:g_super}), it is sufficient to prove
  $G_\lambda=g_\lambda$ for the $\beta_1=\beta_2=\ldots=1$
  substitution. For that substitution, both $g_\lambda$ and
  $G_\lambda$ vanish if the number of $\alpha$'s is less then the
  length of $\lambda$ (see Proposition \ref{prop:G vanish} and Lemma
  \ref{lem:g_vanish}).

  Let $\lambda=(\lambda_1,\ldots,\lambda_r)$ and consider formula
  (\ref{eqn:defg}) for $k=r, l=0$. We will apply the Residue Theorem
  for each residue $\Res_{z_i=0,\infty}$, i.e. we replace
  $\Res_{z_i=0,\infty}$ by $-\sum_p \Res_{z_i=p}$ where sum runs over
  all poles different from $0$ and $\infty$. We claim that the only
  such poles are at $z_i=1/\alpha_j$. Indeed the substitution
  $\beta_i=1$ makes the exponent of $(1-z_i)$ in the formula equal to
  $\lambda_i-i+r $, which is nonnegative.

  The only nonzero finite residues hence correspond to permutations
  $\sigma \in S_r$: $z_i=1/\alpha_{\sigma(i)}$. Straightforward calculation shows that the
  $-\Res_{z_i=1/\alpha_{\sigma(i)}}$ operation yields the term
  corresponding to $\sigma\in S_r$ in Proposition~\ref{prop:many_vars}. This proves the theorem.
\end{proof}

\subsection{Consequences of the $g=G$ theorem.}\label{sec:ResConsequences}
\gr polynomials have a rich algebraic structure and they display
beautiful finiteness and alternating-sign properties. We believe that
the residue form for the stable \gr polynomials above sheds light on
many of those properties. We will illustrate this in Section
\ref{sec:A2} in a so-far unexplored situation---the Thom polynomials
of singularities. Here we will just sketch a simple example showing
how the multiplication structure of \gr polynomials is encoded in
their residue form.

\subsection{Multiplication}
Consider the concrete example of calculating the $g$-expansion of the
product $g_2 \cdot g_2$ (here ``2'' in the subscript is a length 1
partition). We have
$$g_2\cdot g_2=\Res_{z=0,\infty} (1-z)M(z) \cdot \Res_{u=0,\infty} (1-u)M(u)=\Res_{z,u=0,\infty}
(1-z)(1-u)M(z,u)=$$
$$\Res_{z,u=0,\infty} \left( (1-z)(1-u)\frac{1}{1-\frac{u}{z}} \left(1-\frac{u}{z}\right)M(z,u)\right)=$$
$$\Res_{z,u=0,\infty} \left( (1-z)(1-u) \left( \sum_{i=0}^2 \frac{(1-z)^i}{(1-u)^{i+1}} -
    \sum_{i=1}^2 \frac{(1-z)^{i}}{(1-u)^{i}} +
    \frac{u(1-z)^3}{(z-u)(1-u)^3} \right)
  \left(1-\frac{u}{z}\right)M(z,u)\right).
$$
The term involving ${u(1-z)^3}/{((z-u)(1-u)^3)}$ has $u$-residue 0, because of
Lemma \ref{lem:residue0}.
Hence we further obtain
$$g_2\cdot g_2=\Res_{z,u=0,\infty} \left(\left( \sum_{i=0}^2 \frac{(1-z)^{i+1}}{(1-u)^i} - \sum_{i=1}^2 \frac{(1-z)^{i+1}}{(1-u)^{i-1}}\right)\left(1-\frac{u}{z}\right) M(z,u) \right)$$
$$=g_{2,2}+g_{3,1}+g_{4,0}-g_{3,2}-g_{4,1}.$$

\smallskip

In general the calculation of products of arbitrary \gr polynomials is
similar, see \cite{AR}. Namely, to find an explicit expression for $g_I \cdot g_J$ as sums of \gr
polynomials, one considers
$$\prod_i (1-z_i)^{I_i-i} \prod_j (1-u)^{J_j-j} \prod_{i,j} \frac{1}{1-\frac{u_j}{z_i}},$$
and replaces $1/(1-u_j/z_i)$ with an appropriate initial sum of its
Laurent series at $z_i=u_j=1$. This should be done in such a way that
the remainder multiplied by
$\prod (1-z_i)^{I_i-i} \prod (1-u)^{J_j-j}$ has 0 residue.

\begin{remark}
  The consideration above shows that the product of two \gr
  polynomials (parametrized by integer sequences) is a {\em finite
    sum} of \gr polynomials parametrized by integer sequences with
  coefficients with {\em alternating signs}. Proving the analogous
  statement for \gr polynomials parametrized by {\em partitions}
  needs extra considerations (cf. \cite{AR}). We will perform a
  similar analysis for Thom polynomials in Section \ref{sec:pos}.
\end{remark}

\section{Fundamental class in cohomology and K-theory} \label{sec:fundclass}

\subsection{The cohomology fundamental class}
Let $X$ be a subvariety of codimension $d$ in a smooth projective
variety $M$. Then $X$ has a well-defined fundamental class
$[X]\in H^{2d}(M,\Q)$, satisfying
\begin{equation}
  \label{eq:pdual}
  \int_X\iota^*\omega=\int_M[X]\cdot\omega,
\end{equation}
where $\iota:X\to M$ is
the embedding, and $\omega\in H^*(M,\Q)$ is arbitrary, cf. \cite{GH}.

There is a natural extension of this notion to the equivariant
setting, which plays a fundamental role in enumerative geometry.  Let
$V$ be a complex vector space acted upon by a complex torus $T$. Then
a $T$-invariant affine subvariety $X$ has a fundamental class
$[X]_T\in H^{2d}_T(V)= H^{2d}_T(\mathrm{pt})$, $d=\mathrm{codim}(X)$, which satisfies the
equivariant version of \eqref{eq:pdual}:
\[
  \int_X\iota^*\omega=\int_V[X]_T\cdot\omega,
\]
where $\omega$ is any equivariantly closed, compactly supported form on $V$.

There is a number of
definitions of this notion (cf. \cite[\S3]{bsz} for a discussion);
below we recall one  due to Joseph \cite{joseph}.
We begin with introducing some necessary notation.

\begin{itemize}
\item Let $\exp:\Lie(T)\to T$ be the exponential map; the pull-back of
  a function from $f:T\to \C$ to $\Lie(T)$ via this map will be denoted
  by $\exp^*f$.
\item For a character $\alpha\in\Hom(T,\C^*)$, we will
write $\bar\alpha$ for the corresponding weight in the weight lattice
$\Wl_T\subset\Ltd$. We will thus have the following equality of
functions on $\Lie(T)$:
\[       \exp^* \alpha = e^{\bar\alpha},
\]
where factor of $2\pi i$ is considered to be absorbed in the
definition of the exponential, and will be ignored in what follows.
\item Fix a $\Z$-basis $\beta_1,\ldots,\beta_r:T\to\C^*$ of
  $\Hom(T,\C^*)$. We then have
\[    H_T^*(V)=H_T^*(\mathrm{pt})=\Z[\bar{\beta}_1,\ldots,\bar{\beta}_r].
\]
\item Let $x_j$, $j=1,\dots N$ be a set of coordinates on $V$,
  corresponding to a basis of eigenvectors of the $T$ action, and
  denote by $\eta_j\in \Hom(T,\C^*)$, $j=1,\dots N$, the corresponding
  characters: for $t\in T$, we have $ t\cdot x_j = \eta_j(t)^{-1} x_j$. For
  what follows, it is convenient to make the following
\begin{assumption}\label{assump}
  All the weight vectors of the vector space $V$ lie in an open half-space
  of the weight lattice $\Wl_T\subset\Ltd$, i.e. there exists an element $Z\in
\Lie(T)$ such that we have
\[        \langle \bar\eta_j,Z\rangle>0, \quad j=1,\dots N.
\]
\end{assumption}
One can carry out the constructions of the theory without this
assumption as well, but this is more technical, and this case is
sufficient for our purposes.
\end{itemize}

Recall that for a finite-dimensional representation $W$ of $T$ with a diagonal
basis
\[    W=\oplus_{i=1}^m\C w_i,\; t\cdot w_i=\alpha_i(t)\cdot w_i, \text{ we have }
\tr \left[t\,|\,W\right]=\sum_{i=1}^m\alpha_i,\text{ for } t\in T.
\]
This function on $T$ is called the {\em character of } $W$.

Now let $X\subset V$ be a $T$-invariant subvariety, and denote by $RX$ the
ring of algebraic functions on $X$. The character
\[  \chi_X(t) = \tr [t\,|\,RX],\quad t\in T
\]
of $RX$ considered as a $T$-representation is only a formal series
since $RX$ is infinite-dimensional whenever the dimension of $X$ is
positive. Under Assumption \ref{assump}, however, this series
converges in a domain in $T$, and $ \chi_X(t) $ makes sense as a
rational function on $T$.

For example, $RV=\C[x_1,\dots,x_N]$ is the
ring of polynomial functions on $V$, and we have
\begin{equation}
  \label{eq:rv}
  \chi_{V} = \prod_{j=1}^N\frac1{1-\eta^{-1}_j},
\end{equation}
as can be seen by expanding this function in an appropriate domain in $T$.

The following theorem is a consequence of the Hilbert's syzygy
theorem (cf. also \cite[\S4.3]{ms}).

\begin{theorem}
  \label{hilbertsyzygy} Let $X\subset V$ be a $T$-invariant subvariety
  of codimension $d$. Then $\chi_{X}$ is a function on $T$ defined
  whenever $\chi_V$ is defined (cf.
  \eqref{eq:rv}), and has the form of a finite integral linear combination
of $T$-characters multiplied by $\chi_{V}$:
  \begin{equation}
    \label{eq:chirx}
  \chi_{X} = \chi_{V}\cdot \sum_{j=1}^M a_j \theta_j,\,\text{ where
} a_j\in\Z,\,\theta_j\in \Hom(T,\C^*).
  \end{equation}
  Moreover, expanding the function $\exp^*(\chi_X/\chi_V)=
  \sum_{j=1}^M a_j \bar\theta_j$ on $\mathrm{Lie}(T)$
  around the origin, we obtain a power series with
  lowest degree terms in degree $d$:
\begin{equation}
  \label{eq:expleading}
   \sum_{j=1}^M a_j \exp\bar\theta_j = \frac{1}{d!}\sum_{j=1}^M a_j
\bar\theta^d_j +\rho_{d+1}\;\text{ with }\rho_{d+1}\in\mathfrak{m}^{d+1},
\end{equation}
where $\mathfrak{m}$ is the maximal ideal of analytic functions vanishing at
the origin in $\Lie(T)$.
\end{theorem}

The last part of the theorem states that, after the expansion, the
terms up to degree $d-1$ cancel.
\begin{definition}
  Let $X\subset V$ be a $T$-invariant subvariety of codimension $d$. We
  define the {\em $T$-equivariant fundamental class}  of $X$ in $V$
  as the degree-$d$ (leading) term
  on the right hand side of \eqref{eq:expleading} interpreted as an
  element of $H^*_T(V)$:
\[      [X]_T = (-1)^d\sum_{j=1}^M a_j \bar\theta^d_j.
\]
\end{definition}

\begin{example}\label{exaffine} \rm
  Let $V=\C^2$ be endowed with a diagonal action of $T=\C^*$ with
  weight 1 on each of the two coordinate functions $x$ and $y$, and
  let $X =\{xy=0\}$. Then $X$ is $T$-invariant, and there is a short
  exact sequence of  $RV$-modules
\[    0\to   RV[2]\to RV\to RX\to 0,
\]
where $RV[2]$ stands for the free module of rank 1, generated by a
single element of degree 2, whose image is the function $xy$.  This
implies
\[ \chi_{V}=\frac1{(1-\beta^{-1})^2},\quad\text{and}\quad
\chi_{X}=\frac{1-\beta^{-2}}{(1-\beta^{-1})^2}=\frac{1+\beta^{-1}}{1-\beta^{-1}}.
\]
Now we substitute $\beta=e^{\bar{\beta}}$, and we see that modulo ${\bar{\beta}}^3$, we
have $\chi_{X}/\chi_{V} = 1-\beta^{-2} = 2\bar{\beta}$, and hence $[X]_T=2\bar{\beta}$.
\end{example}

\subsection{Equivariant K-theoretic fundamental classes}
It is not immediately obvious what one should take as the appropriate
definition of the equivariant fundamental class in $K$-theory.

In our setup, we have
\[ K_T(\mathrm{pt})=\Z\Hom(T,\C^*)=\Z[\beta_1^{\pm 1}, \beta_2^{\pm 1},\dots,\beta_r^{\pm 1}],
\]
and thus for a $T$-invariant $X\subset V$, it would seem natural
to define as this fundamental class the linear combination of torus characters
$\chi_{X}/\chi_{V}$ in \eqref{eq:chirx}, which naturally lies in
this space.\footnote{This polynomial is called the $K$-polynomial in
  \cite{ms} for this reason.}
This invariant is very difficult to calculate, however
(cf. \cite{kolokol} for a more detailed discussion), and, in fact,
there are some alternatives.

\begin{proposition} \label{proplead}
 Let $X\subset V$ be a $T$-invariant subvariety in the vector space
 $V$ satisfying Assumption \ref{assump}. Then the cohomology groups of the structure sheaf
 $H^i(Y,\calO_Y)$ for a  smooth $T$-equivariant resolution
 $\pi:Y\to X$ are independent of the choice of $Y$, and thus are
 invariants of $X$. In particular,
 \begin{equation}
\label{eq:pushchar}
\chik_X(\tau)\overset{\mathrm{def}}{=} \sum_{i=0}^{\dim Y}(-1)^i\tr
\left[\tau\,|\,H^i(Y,\calO_Y)\right]
\end{equation}
is an invariant of $X$, which coincides with $\chi_{X}$ if $X$ has
only rational singularities. Moreover, $\chi_{X}/\chi_{V}$ and
$\chik_{X}/\chi_{V}$ have the same leading term in the sense of
\eqref{eq:chirx} and \eqref{eq:expleading} in Theorem \ref{hilbertsyzygy}.
\end{proposition}

These statements are fairly standard---see for example
\cite{ms,Hsh}---hence we only give a {\em sketch of the proof} to
emphasize the key ideas involved. First we recall that for two smooth
resolutions $Y_1\to X\leftarrow Y_2$, there exists a resolution
$Y\to X$ which dominates $Y_1,Y_2$. This fact reduces the theorem to
the case when both $X$ and $Y$ are smooth and $\pi$ is birational. In
this case, the first statement may be found in \cite[Chapter III]{Hsh}.

The  statement on rational singularities is essentially a tautology: for an affine
variety $X$, having rational singularities means precisely that for
any smooth resolution $Y\to X$, we have
$H^0(Y,\calO_Y)=H^0(X,\calO_X)$ and $H^i(Y,\calO_Y)=0$ for $i>0$.

Finally, note that the cohomology groups $H^i(Y,\calO_Y)$ are the
sections over $X$ of the derived push-forward sheaves
$R^i\pi_*\calO_Y$. Applying the flat base change for the smooth locus
in $X$, we see that for $i>1$, these sheaves are supported on the
singular locus of $X$, which is of higher codimension than $X$
itself. For such a sheaf then, the corresponding leading term will be
of higher degree than $d$, the codimension of $X$ (see \cite{ms}), and this
completes the proof. \qed

\begin{definition} \label{def-Kclass}
  Let $X$ be a $T$-invariant subvariety of the vector space $V$
  endowed with a $T$-action and satisfying Assumption
  \ref{assump}. Then we define the $K$-theoretic fundamental class $[X]^K_T$ of
  $X$ in $V$ as the character $\chik_X/\chi_V$, where $\chik_X$ is
  given by the formula \eqref{eq:pushchar}.
\end{definition}

Now let us revisit Example \ref{exaffine}.
Denote by $Y$ the normalization of   $X$, which is the union of two nonintersecting lines. Then $H^0(Y,\calO_Y)$ is two copies of a polynomial ring in one variable, and $H^0(X,\calO_X)\subset H^0(Y,\calO_Y)$ is the subset of those pairs of polynomials whose constant terms coincide. We have
\[
\chik_X=\chi_Y=\frac{2}{1-\beta^{-1}}, \quad
\chi_V=\frac{1}{(1-\beta^{-1})^2}, \quad
\text{and hence} \quad
[X]_T^K=\frac{\chik_X}{\chi_V}=2(1-\beta^{-1}).
\]

 It is instructive to verify directly the last statement of
  Proposition \ref{proplead} even in this simple case. When we used
  $\chi_X$ instead of $\chik_X$, we obtained a different
  answer:
\[
\frac{\chi_X}{\chi_V}=\frac{(1+\beta^{-1})/(1-\beta^{-1})}{1/(1-\beta^{-1})^2}=1-\beta^{-2}.
\]
Yet, after substituting $\beta=e^{\bar{\beta}}$, we see that, modulo $({\bar{\beta}}^3)$ 
we have the equality: \[
\chik_{X}/\chi_{V} = \chi_{X}/\chi_{V} = 2\bar{\beta}\quad \mod ({\bar{\beta}}^3), \] 
 recovering the cohomological fundamental class of Example \ref{exaffine}.

\begin{remark}
\rm For a
  holomorphic map between complex manifolds $g:M^a\to P^b$, one can
  consider the ${\eta}$-singularity points
\[
{\eta}(g)=\{ x\in M : \text{the $N$-jet of $g$ at $x$ belongs to ${\eta}$}\}.
\]
Thom's principle on cohomological Thom polynomials states that if $g$
satisfies certain transversality properties then
\[
[{\eta}(g)]=\Tp^{a\to b}_\eta(\text{Chern roots of }TM,\text{Chern roots of }g^*(TP)).
\]
This powerful statement relies on the fact that the notion
of ``cohomological fundamental class'' is {\em consistent with pullback morphisms}. 
The way we set up the notion of K-theoretic fundamental
class in Definition \ref{def-Kclass} is {\em not} consistent with pullback morphisms (rather, it is
consistent with push-forward morphisms), hence Thom's principle does
not hold for our K-theoretic Thom polynomials. The interesting project
of studying another version of K-theoretic fundamental class of
singularities---one for which Thom's principle holds---is started in
\cite{kolokol}.
\end{remark}

We end this section with an observation addressing the situation when the
group $G$ acting on $V$ is a general reductive group with maximal torus
$T$.  For a reductive group $G$, we have $K_G(\pt)=K_T(\pt)^W$
(the Weyl-invariant part). For a $G$-invariant $X\subset V$, the class $[X]_T^K$
will be in this Weyl-invariant part, and hence we can define
$[X]_G^K = [X]_T^K$.

\smallskip

In the rest of the paper, if the group that acts is obvious, we will
drop the subscript and use the notation $[X]=[X]_G$, $[X]^K=[X]_G^K$
for the cohomological and K-theoretic fundamental class.

\section{Singularities and their Thom polynomials}

 Recall the notion of contact singularities and their Thom polynomials
 from \S\ref{sec:introktp}. 
Let us see a few examples.

\begin{example} \rm
\
\begin{itemize}
\item The simplest case is $Q=\C$, also known as the
  $\A_0$-algebra. In this case, we have
  \[
\eta_{\A_0}^{a\to b}=J_N(\Ce,\Cb),
\] which is essentially the
  inverse function theorem.
\item When the algebra $Q$ is $\A_1=\C[x]/(x^2)$, the set
  $\eta_{\A_1}^{a\to b}$ is the set of singular map-jets, i.e.
  those whose derivative at 0 is not injective.
\item For $r>0$, consider
  $Q=\C[x_1,\ldots,x_r]/(x_1,\ldots,x_r)^2$. In this case,
  $\eta_{Q}^{a\to b}$ is the set of those map-jets whose linear part
  has corank at least $r$ (also known as the $\Sigma^r$ singularity).
\item The contact singularities corresponding to the algebra $Q=\A_r=\C[x]/(x^{r+1})$
 are called {\em Morin singularities}. A generic element of
  $\eta_{\A_2}^{2\to2}$ may be represented as $(x,y)\mapsto (x^3+xy, y)$; it is called the
  cusp singularity.
 \end{itemize}
\end{example}

\subsection{The model}

By a {\em model} for a singularity $\eta \subset J(\Ce,\Cf)$, we mean
a $\GL(\Ce) \times \GL(\Cf)$-equivariant commutative diagram
\begin{equation*}
\xymatrix{
X \ar[dr]_{\pi} \ar[r]_{i\ \ \ \ \ \ \ } \ar@/^2pc/[rr]^\rho & M \times  J(\Ce,\Cf) \ar[d]_{\pi_1} \ar[r]_{\ \ \ \ \ \pi_2}&  J(\Ce,\Cf) \ar[d] \\
   & M  \ar[r]^{p_M}              & pt,
}
\end{equation*}
where
\begin{itemize}
\item $M$ is a smooth compact manifold,
\item  $\pi:X \to M$ is a subbundle of the trivial bundle $\pi_1: M
  \times  J(\Ce,\Cf) \to M$,
\item $\rho=\pi_2 \circ i$ is birational to ${\eta}$,
\item and $p_M$ is the map from $M$ to a point $pt$.
\end{itemize}
   
Let $\nu$ be the quotient bundle of $\pi_1: M\times  J(\Ce,\Cf) \to M$ by $X \to M$.
It follows that for such a model for the singularity $\eta$ one has
$$\Tp_{\eta} = p_{M*} (e(\nu)),$$
where $e$ stands for the (equivariant) Euler class. Indeed, we have
\begin{equation}\label{thearg}
\Tp_{\eta}=\rho_*(1)=\pi_{2*}( i_* (1))=\pi_{2*}( e(\nu) ) = p_{M*} (e(\nu)).
\end{equation}

The advantage of our definition of K-theoretic fundamental class in Section \ref{sec:fundclass} is that the argument (\ref{thearg}) goes through without change to the K-theoretic setting, and we have
$$\KTp_{\eta} = p_{M!}( e(\nu) ),$$
where $e$ is now the K-theoretic (equivariant) Euler class, and $p_{M!}$ is the K-theoretic push-forward map.

\subsection{Integration in K-theory using residues} 

In what follows we will use
residue calculus for the push-forward map in K-theory. 

Let the torus $T$ act on the smooth variety $X$ with finitely many
fixed points. Let $W$ be a rank-$d$ equivariant vector bundle over
$X$, and let $\omega_1,\ldots,\omega_w$ be its Chern roots (i.e.
virtual line bundles whose sum is $W$). Let $p: \Gr(r,W)\to X$ be the
Grass\-mannization of $W$, that is an equivariant bundle whose fiber
over $x\in X$ is the Grassmannian $\Gr(r,W_x)$ of dimension $r$ linear
subspaces of the fiber $W_x$ of $W$ over $x$.  Let $S$ be the
tautological subbundle over $\Gr(r,W)$, and let
$\sigma_1,\ldots,\sigma_r$ be its Chern
roots. 
A symmetric Laurent polynomial $g(\sigma_1,\ldots,\sigma_r)$ is hence
an element of $K_T(\Gr(r,W))$.

\begin{lemma} \label{lem:intres}
We have
\begin{equation}\label{eqn:intres}
p_! ( g(\sigma_1,\ldots,\sigma_r)) =
\mathop{\Res}_{z_1=0,\infty} \ldots
\mathop{\Res}_{z_r=0,\infty}
\left(
\prod_{i>j} \left( 1-\frac{z_i}{z_j}\right)
\frac{g(z_1,\ldots,z_r)}{\prod_{i=1}^r
\prod_{j=1}^w \left( 1- \frac{z_i}{\omega_j}\right)}
\prod_{i=1}^r\frac{dz_i}{z_i} \right).
\end{equation}
\end{lemma}

\begin{proof}
Consider first the special case when $X$ is a point. Then the equivariant localization formula for the push-forward map is
\[
p_!( f(\sigma_1,\ldots,\sigma_r) )=
\sum_{I}
\frac{ f(\omega_{I_1},\ldots,\omega_{I_r})}{\prod_{i\in I} \prod_{j\in \bar{I}} \left( 1-\frac{\omega_i}{\omega_j}\right)},
\]
where the summation is over $r$-element subsets $I$ of $\{1,\ldots,n\}$, and $\bar{I}$ is the complement of~$I$.
Applying the Residue Theorem for the right hand side of (\ref{eqn:intres}), for $z_1, z_2, \ldots$ gives the same expression. This proves the lemma when $X$ is a point.

The general case is shown applying this special case to $W$ restricted to fixed points.
\end{proof}

When $G$ is a connected algebraic group $G$, Lemma~\ref{lem:intres}
may be applied to the maximal torus $T\subset G$, and since $K_G(X)$ is
the Weyl-invariant part of $K_T(X)$, formula (\ref{eqn:intres}) holds
without change.

\section{$\Sigma^r$ singularities}

\subsection{The model for $\Sigma^r$}

The obvious model for the
\[
\Sigma^r=\Sigma^r(\Ce,\Cf)=\{g\in J^1(\Ce,\Cf): \dim \ker g \geq r\}
\]
singularity is $M=\Gr(r,\Ce)$, and
\[
X=\{(V,g)\in \Gr(r,\Ce)\times J^1(\Ce,\Cf): g|_V=0\}.
\]
Let the tautological rank $r$ bundle over $\Gr(r,\Ce)$ be $S$. The bundle $\pi: X \to \Gr(r,\Ce)$ can be identified with $J^1(\Ce/S,\Cf)$, hence the normal bundle is $\nu=J^1(S,\Cf)$. Thus $\KTp_{\Sigma^r}=p_!( e(J^1(S,\Cf) )) $ for the map $p:\Gr(r,\Ce)\to \pt$.

\begin{theorem}
We have
\begin{equation}\label{eqn:sigmar_res}
\KTp_{\Sigma^r}=
\mathop{\Res}_{z_1=0,\infty} \ldots
\mathop{\Res}_{z_r=0,\infty}
\left(
\prod_{i>j} \left( 1-\frac{z_i}{z_j}\right)
\prod_{i=1}^r
\frac{ \prod_{j=1}^b \left( 1- \frac{z_i}{\beta_j}\right)}
{\prod_{j=1}^a \left( 1- \frac{z_i}{\alpha_j}\right)}
\prod_i\frac{dz_i}{z_i} \right).
\end{equation}
\end{theorem}

\begin{proof} We have
\[
\KTp_{\Sigma^r}=p_! (  e(J^1(S,\Cf)) )=
p_!
\left(  \prod_{i=1}^r \prod_{j=1}^b \left( 1-\frac{\sigma_i}{\beta_j} \right)\right),
\]
and applying Lemma \ref{lem:intres} proves the Theorem.
\end{proof}

Comparing expression (\ref{eqn:sigmar_res}) with the residue formula for \gr polynomials (Definition~\ref{def:g}), we obtain
$$\KTp_{\Sigma^r}=G_{(r+l)^r} (\ep_1^{-1},\ldots,\ep_a^{-1}; \beta_1^{-1},\ldots,\beta_b^{-1}).$$
This result is known in Schubert calculus \cite{LS} as the K-theoretic Giambelli-Thom-Porteous formula.

\section{$\A_2$ singularities} \label{sec:A2}

\subsection{The model for $\A_2$}
Consider the tautological exact sequence $S\to \Ce \to Q$ over
$\Gr(1,\Ce)$. Let $M=\Gr(1, S^{\otimes 2} \oplus Q)$ be the
projectivization of the vector bundle $S^{\otimes 2} \oplus Q$ over
$\Gr(1,\Ce)$, and denote the tautological line bundle over $M$ by $D$.

According to \cite{bsz, kaza:noass} there is a model for the
\[
\eta_{\A_2}^{a\to b}=\overline{\{g\in J^2(\Ce,\Cf): Q_g\cong \C[x]/(x^3)]\}}
\]
singularity with this $M$, and  normal bundle $\nu= \Hom(S\oplus D,\Cf)$.

\subsection{Residue formula for $\KTp_{\A_2}$}

\begin{theorem} \label{thm:A2R}
We have
\[
\KTp_{\A_2}^{a\to b}=
\mathop{\Res}_{z_1=0,\infty}
\mathop{\Res}_{z_2=0,\infty}
\left(
\frac{ 1 -\frac{z_2}{z_1} }{ 1-\frac{z_2}{z_1^2} }  \prod_{i=1}^2
\frac{ \prod_{j=1}^b \left( 1- \frac{z_i}{\beta_j}\right)}
{\prod_{j=1}^a \left( 1- \frac{z_i}{\alpha_j}\right)}
\frac{dz_2 dz_1}{z_2 z_1} \right).
\]
\end{theorem}

Note that the order of taking residues is important here: first we take residues with respect to $z_2$, then with respect to $z_1$.

\begin{proof} We know that $\KTp_{\A_2}=p_{M!} (  e(\Hom(D\oplus S,\Cf)) )$. Let the Chern roots of the bundle $Q$ be $\omega_1,\ldots,\omega_{a-1}$, and let the class of $S$ be $\sigma$, and the class of $D$ be $\tau$. We have
\[
e(\nu)=\prod_{j=1}^b \left( 1-\frac{\sigma}{\beta_j} \right)
\prod_{j=1}^b \left( 1-\frac{\tau}{\beta_j} \right).
\]
Pushing forward this class to $\Gr(1,\Ce)$, using Lemma \ref{lem:intres} we get
\[
\mathop{\Res}_{z_2=0,\infty} \left(
\frac{  \prod_j\left( 1-\frac{\sigma}{\beta_j} \right)\prod_j\left( 1-\frac{z_2}{\beta_j} \right)}
{\left( 1-\frac{z_2}{\sigma^2} \right)\prod_j \left( 1-\frac{z_2}{\omega_j} \right)}
\frac{dz_2}{z_2} \right).
\]
Using the fact that $S\to \Ce\to Q$ is an exact sequence, this is further equal to
\[
\mathop{\Res}_{z_2=0,\infty} \left(
\frac{  \prod_j\left( 1-\frac{\sigma}{\beta_j} \right)\prod_j\left( 1-\frac{z_2}{\beta_j} \right)\left( 1-\frac{z_2}{\sigma} \right)}
{\left( 1-\frac{z_2}{\sigma^2} \right)\prod_j\left( 1-\frac{z_2}{\alpha_j} \right)}
\frac{dz_2}{z_2}\right).
\]
Pushing this class further from $\Gr(1,\Ce)$ to a point, using Lemma \ref{lem:intres}, we obtain
\[
\mathop{\Res}_{z_1=0,\infty}
\mathop{\Res}_{z_2=0,\infty} \left(
\frac{  \prod_j\left( 1-\frac{z_1}{\beta_j} \right)\prod_j\left( 1-\frac{z_2}{\beta_j} \right)\left( 1-\frac{z_2}{z_1} \right)}
{\left( 1-\frac{z_2}{z_1^2} \right)\prod_j\left( 1-\frac{z_2}{\alpha_j} \right)\prod_j\left( 1-\frac{z_1}{\alpha_j} \right)}
\frac{dz_2}{z_2}\frac{dz_1}{z_1} \right),
\]
which is what we wanted to prove.
\end{proof}

\subsection{$\KTp_{\A_2}$ in terms of \gr polynomials---the stable expansion}\label{sec:expansions}


Let
\begin{equation*}\label{eqn:series1}
\frac{1}{1-z_2/z_1^2}=\sum_{r,s} d_{r,s} (1-z_1)^r(1-z_2)^s
\end{equation*}
be the Laurent expansion of the named rational function on the $|1-z_1| < |1-z_2|$ region. Equivalently, after substituting $x_1=1-z_1,x_2=1-z_2$, let
\begin{equation*}\label{eqn:series2}
\frac{1-2x_1+x_1^2}{x_2-2x_1+x_1^2}=\sum_{r,s} d_{r,s} x_1^rx_2^s
\end{equation*}
be the Laurent expansion of the named rational function on the $|x_1|< |x_2|$ region.
Based on the calculation
\begin{eqnarray}\label{eqn:g-calc}
\frac{1}{x_2-2x_1+x_1^2} & = & \frac{1}{x_2}\cdot\frac{1}{1-(2x_1-x_1^2)/x_2}=\sum_{k=1}^\infty \frac{1}{x_2^k}(2x_1-x_1^2)^{k-1}\\
\notag
& = & \sum_{k=1}^\infty \sum_{r=k-1}^{2k-2} (-1)^{r-k+1} 2^{2k-2-r} \binom{k-1}{2k-2-r} x_1^rx_2^{-k},
\end{eqnarray}
we have that 
\[
d_{r,s}=(-1)^{r+s+1}
\left(
2^{-2s-2-r}\binom{-s-1}{-2s-r-2}+
2^{-2s-r}\binom{-s-1}{-2s-r-1}+
2^{-2s-r}\binom{-s-1}{-2s-r}\right)
\]
for $r=0,1,\ldots, s=-r-1,\ldots,-\lfloor r/2 \rfloor$. In particular,
{\em the sign of $d_{r,s}$ is $(-1)^{r+s+1}$}.

For the values of $d_{r,s}$ for small (absolute value) $r,s$ see the table in \S\ref{sec:introktp}.

\begin{theorem}[Grothendieck expansion of $\KTp_{\A_2}$: the stable version]
\label{thm:expansion}
Let $l=b-a$, and $N>2l+2$. Then
\begin{equation} \label{eq:ktpst}
\KTp_{\A_2}^{a\to b}=
\sum_{r=0}^N
\sum_{s=-r-1}^{-\lfloor \frac{r}{2} \rfloor}
d_{r,s} G_{r+l+1,s+l+2}(\ep_1^{-1},\ldots,\ep_a^{-1}; \beta_1^{-1},\ldots,\beta^{-1}_b).
\end{equation}
\end{theorem}

Note that for a given $r$, the set of non-zero $d_{r,s}$ coefficients
are exactly those between $s=-r-1$ and $s=-\lfloor r/2 \rfloor$,
hence, in the summation above, $s$ runs through all its relevant values.

\begin{remark} \label{rem:lstab}
Since $N$ may be arbitrarily large in \eqref{eq:ktpst}, it is tempting to phrase Theorem \ref{thm:expansion} informally as
\begin{equation}\label{eqn:informal}
\KTp_{\A_2}^{a\to b}=
\sum_{r,s}
d_{r,s} G_{r+l+1,s+l+2}(\ep_1^{-1},\ldots,\ep_a^{-1}; \beta_1^{-1},\ldots,\beta^{-1}_b).
\end{equation}
This series does not converge, however.
\end{remark}

\begin{proof}
The finite expansion of $1/(1-z_2/z_1^2)$ with respect to $z_1$, around $z_1=1$, with remainder term is
\begin{equation}\label{eqn:finiteseries}
\frac{1}{1-z_2/z_1^2}=\sum_{r=0}^N \left( \sum_s  d_{r,s} (1-z_2)^s\right) (1-z_1)^r + R_N(z_1,z_2)
\end{equation}
where the $s$-summation is finite. A quick calculation shows that the
remainder term may be expressed as
\begin{equation}\label{eqn:remainder}
R_N(z_1,z_2)=-
\left( \frac{1-z_1}{1-z_2}\right)^{N+1}\frac{z_1q_N(z_2)+p_N(z_2)}{1-z_1^2/z_2}.
\end{equation}
where
\[
p_N(z)=\sum_{i=0}^{\lfloor \frac{N+1}{2} \rfloor} \binom{N+1}{2i}z^i,
\qquad
q_N(z)=\sum_{i=0}^{\lfloor \frac{N}{2} \rfloor} \binom{N+1}{2i+1}z^i.
\]
According to Theorem \ref{thm:A2R}, we have the following expression
for $\KTp_{\A_2}$:
\[
\KTp_{\A_2}^{a\to b}=
\mathop{\Res}_{z_1=0,\infty}
\mathop{\Res}_{z_2=0,\infty}
\left(
(1-z_1)^{l}(1-z_2)^{l}
\frac{1}{1-z_2/z_1^2}
\times \right.\hskip 5 true cm\]
\[ \left. \hskip 7 true cm \times
\left( 1 -\frac{z_2}{z_1} \right)  \prod_{i=1}^2
\frac{ \prod_{j=1}^b \left( 1- \frac{z_i}{\beta_j}\right)}
{\prod_{j=1}^a \left( 1- \frac{z_i}{\ep_j}\right) (1-z_i)^l}
\frac{dz_2 dz_1}{z_2 z_1} \right).
\]
Substituting (\ref{eqn:finiteseries}), we obtain
\[
\KTp_{\A_2}^{a\to b}=
\mathop{\Res}_{z_1=0,\infty}
\mathop{\Res}_{z_2=0,\infty}
\left(
\sum_{r=0}^N \left( \sum_s  d_{r,s} (1-z_2)^{s+l}\right) (1-z_1)^{r+l}
\times \right.\hskip 3.5 true cm\]
\[ \left. \hskip 5 true cm \times
\left( 1 -\frac{z_2}{z_1} \right)  \prod_{i=1}^2
\frac{ \prod_{j=1}^b \left( 1- \frac{z_i}{\beta_j}\right)}
{\prod_{j=1}^a \left( 1- \frac{z_i}{\ep_j}\right) (1-z_i)^l}
\frac{dz_2 dz_1}{z_2 z_1} \right)+
\]
\[
\mathop{\Res}_{z_1=0,\infty}
\mathop{\Res}_{z_2=0,\infty}
\left(
R_N(z_1,z_2)
\left( 1 -\frac{z_2}{z_1} \right)  \prod_{i=1}^2
\frac{ \prod_{j=1}^b \left( 1- \frac{z_i}{\beta_j}\right)}
{\prod_{j=1}^a \left( 1- \frac{z_i}{\ep_j}\right) }
\frac{dz_2 dz_1}{z_2 z_1} \right).
\]
According to the residue expression for \gr polynomials (Definition
\ref{def:g}) the first term equals
\[
\sum_{r=0}^N \sum_s d_{r,s}
G_{r+l+1,s+l+2}(\ep_1^{-1},\ldots,\ep_a^{-1},\beta_1^{-1},\ldots,\beta_b^{-1}),
\]
and we claim that the second term vanishes for large $N$. Indeed, using
the form (\ref{eqn:remainder}) of the remainder term $R_N(z_1,z_2)$,
we can see that for large $N$, the rational form
\begin{equation}\label{eqn:maradek}
R_N(z_1,z_2)\left( 1 -\frac{z_2}{z_1} \right)  \prod_{i=1}^2
\frac{ \prod_{j=1}^b \left( 1- \frac{z_i}{\beta_j}\right)}
{\prod_{j=1}^a \left( 1- \frac{z_i}{\ep_j}\right) }
\frac{dz_2 dz_1}{z_2 z_1}
\end{equation}
satisfies the conditions of Lemma \ref{lem:residue0} in $z_2$. This
means that already applying the first residue operation
$\Res_{z_2=0,\infty}$ 
results in 0. This completes the proof.
\end{proof}

\subsection{$\KTp_{\A_2}$ in terms of \gr polynomials -- the minimal expansion}\label{sec:min}

\begin{theorem}[\gr expansion of $\KTp_{\A_2}$, the minimal version]
\label{thm:expansion2}
We have the following expression for $\KTp_{\A_2}^{a\to b}$ in \gr
polynomials indexed by partitions:
\begin{equation*}
\KTp_{\A_2}^{a\to b}=
\sum_{r=0}^{2l+2}
\sum_{-l-2}^{-\lfloor \frac{r}{2} \rfloor} D_{r,s,l} \cdot G_{r+l+1,s+l+2}(\ep_1^{-1},\ldots,\ep_a^{-1}; \beta_1^{-1},\ldots,\beta^{-1}_b),
\end{equation*}
where $l=b-a$, and 
\[
D_{r,s,l}=\begin{cases}
d_{r,s} & \text{if } s>-l-2 \\
\sum_{t=-r-1}^{-l-2} d_{r,t}=\sum_{t=-\infty}^{-l-2} d_{r,t} & \text{if } s=-l-2.
\end{cases}
\]
\end{theorem}

\begin{proof}
It follows from Theorem \ref{thm:expansion}  that for large $N$
\begin{equation}\label{eqn:start}
\KTp_{\A_2}^{a\to b}=
\sum_{r=0}^N
\sum_{s=-r-1}^{-\lfloor \frac{r}{2} \rfloor}
d_{r,s} G_{r+l+1,s+l+2}.
\end{equation}
For notational simplicity we omit the arguments $\ep_i^{-1},\beta_i^{-1}$ of the \gr polynomials.
Consider the sum
\[
\sum_{s=-r-1}^{-\lfloor r/2 \rfloor} d_{r,s} G_{r+l+1,s+l+2}
\]
for a given $r$. In it, the occurring \gr polynomials have the same first index $r+l+1$, but varying second index $s+l+2$.
Notice that if $r>2l+2$ then all $s+l+2$ indexes 
are non-positive. Indeed, if $r>2l+2$ then $s\leq -\lfloor r/2 \rfloor < -\lfloor (2l+2)/2 \rfloor =-l-1$ and hence $s+l+2 < 1$. Then using the straightening law $G_{I,0}=G_{I,-1}=G_{I,-2}=\ldots$ (see (\ref{eqn:straightening2}) or Lemma \ref{lem:replace}) we have that
\begin{equation}\label{eqn:vanish}
\sum_{s=-r-1}^{-\lfloor r/2 \rfloor} d_{r,s} G_{r+l+1,s+l+2}=
\left(\sum_{s=-r-1}^{-\lfloor r/2 \rfloor} d_{r,s}\right) G_{r+l+1,0}.
\end{equation}
Plugging in $z_2=0$ into $1/(1-z_2/z_1^2)$ results 1, hence for $r>0$
we have $\sum_{s=-r-1}^{\lfloor r/2 \rfloor} d_{r,s}=0$, and in turn,
the expression (\ref{eqn:vanish}) is 0. This proves that in
(\ref{eqn:start}) the number $N$ can be chosen to be as small as
$2l+2$. The same statement may be obtained from a careful analysis of
the vanishing of the residues of (\ref{eqn:maradek}).

Now let $r\leq 2l+2$. Using the same straightening law of \gr polynomials we obtain
\[
 \sum_{s=-r-1}^{-\lfloor r/2 \rfloor} d_{r,s} G_{r+l+1,s+l+2}=
 \underbrace{ \left(  \sum_{s=-r-1}^{-l-2} d_{r,s} \right)}_{D_{r,s,l}} G_{r+l+1,0} +
 \sum_{s=-l-1}^{-\lfloor r/2 \rfloor} d_{r,s} G_{r+l+1,s+l+2},
\]
completing the proof.
\end{proof}

\begin{remark}
The expansion in Theorem \ref{thm:expansion2} is minimal in the sense that each occurring \gr polynomial is parametrized by a partition (with non-negative components), and hence can not be simplified by the straightening laws (\ref{eqn:straightening1})-(\ref{eqn:straightening2}) (or Lemma \ref{lem:replace}).
\end{remark}

\section{Alternating signs}\label{sec:pos}

The coefficients of the \gr polynomials in both the {\em stable} and the {\em minimal} \gr polynomial expansions of $\KTp_{\A_2}$ have alternating signs.

\begin{theorem}
The coefficient of $G_{a,b}(\ep_1^{-1},\ldots,\ep_{e}^{-1};\beta_1^{-1},\ldots,\beta_b^{-1})$ in both the expansion of Theorem \ref{thm:expansion} and the expansion of Theorem \ref{thm:expansion2} has sign $(-1)^{a+b}$.
\end{theorem}

\begin{proof}
The statement for the expansion in Theorem \ref{thm:expansion} is equivalent to $d_{r,s}$ having sign $(-1)^{r+s+1}$, which follows from the explicit formula for $d_{r,s}$ in Section \ref{sec:expansions}.

The statement for the expansion in Theorem \ref{thm:expansion2} is equivalent to $D_{r,s,l}$ having sign $(-1)^{r+s+1}$ for any $l$. For this we need to additionally prove that
\begin{equation} \label{eqn:sign}
\text{the sign of }
\sum_{t=-\infty}^{-l-2} d_{r,t}
\text{ is }
(-1)^{r+s+1}
\end{equation}
for any $l$.

To prove (\ref{eqn:sign}) consider $f=(1-2x_1+x_1^2)/(x_2-2x_1+x_1^2)=\sum_{r,s}d_{r,s} x_1^rx_2^s$ (as before, $|x_1|<|x_2|$), and  let $g=(-1+f)/(1-x_2)$.
On the one hand $g=1/(x_2-2x_1+x_1^2)$ (from the explicit form of $f$). On the other hand
\[
g=\left(-1+\sum_{r,s} d_{r,s} x_1^rx_2^r\right)(1+x_2+x_2^2+\ldots)=\sum_{r,s}\left( \sum_{t=-\infty}^s d_{r,t} \right) x_1^rx_2^s.
\]
Here we used that $d_{0,-1}=1$ and $d_{0,s}=0$ for all $s\not=-1$.  

Comparing the two forms of $g$ we find that statement (\ref{eqn:sign}) is equivalent to the the property that the coefficient of $x_1^rx_2^s$ in the expansion of $1/(x_2-2x_1+x_1^2)$ has sign $(-1)^{r+s+1}$. This latter claim follows from the calculation (\ref{eqn:g-calc}).
\end{proof}

\section{Remarks on higher singularities}\label{sec:higher}

For singularities higher than $\A_2$, it is difficult to carry out our
program. There are no practical models for $\A_d$-singularities for
$d\ge7$, but even in the case of $\A_3$, where the model is very
simple (\cite{bsz,
  kaza:noass}), the combinatorial problems we face are rather complicated.
A proof analogous to that of Theorem \ref{thm:A2R} in this case yields
the following statement.

\begin{theorem}
We have
\[
\KTp_{\A_3}^{a\to b}=
\mathop{\Res}_{z_1=0,\infty}
\mathop{\Res}_{z_2=0,\infty}
\mathop{\Res}_{z_3=0,\infty}
\left(
\frac{ \left(1 -\frac{z_2}{z_1}\right)\left(1 -\frac{z_3}{z_1}\right)\left(1 -\frac{z_3}{z_2}\right) }
{ \left( 1-\frac{z_2}{z_1^2} \right) \left( 1-\frac{z_3}{z_1^2} \right)  \left( 1-\frac{z_3}{z_1z_2} \right)  }
\prod_{i=1}^3
\frac{ \prod_{j=1}^b \left( 1- \frac{z_i}{\beta_j}\right)}
{\prod_{j=1}^a \left( 1- \frac{z_i}{\alpha_j}\right)}
\frac{dz_3 dz_2 dz_1}{z_3 z_2 z_1} \right).
\] \qed
\end{theorem}

This formula suggests that to obtain the \gr expansion of
$\KTp_{\A_3}$, we ought to consider the expansion
\[
\frac{1}
 { \left( 1-z_2/z_1^2 \right) \left( 1-z_3/z_1^2\right)  \left( 1-z_3/z_1z_2 \right)  }
 =
 \sum_{r,s,t} d_{r,s,t} (1-z_1)^r(1-z_2)^s(z-z_3)^t,
\]
valid in the region
$|1-z_1|<|1-z_2|<|1-z_3|$, and then find an appropriate way to resum
the series
\begin{equation}\label{eqn:expansion3}
\sum_{r,s,t} d_{r,s,t}
G_{r+l+1,s+l+2,t+l+3}(\ep_1^{-1},\ldots,\ep_a^{-1};
\beta_1^{-1},\ldots,\beta_b^{-1}).
\end{equation}
to obtain finite expressions. The concrete form of the resummation
procedure and the resulting finite expression is not clear at the moment.

It seems even more difficult to find the analogue of Theorem
\ref{thm:expansion2} (the minimal \gr expansion) for $\A_3$. To
achieve the \gr expansion of Theorem \ref{thm:expansion2} from that of
Theorem \ref{thm:expansion} we needed to work only with {\em one} of
the straightening laws, namely (\ref{eqn:straightening2}). However, to
``straighten'' the partitions in (\ref{eqn:expansion3}) one is forced
to use the other straightening law, namely (\ref{eqn:straightening1}),
and this  seems much more complex. It would be
interesting to develop the residue calculus or another analytic tool
which replaces the combinatorics of (\ref{eqn:straightening1}).

\end{document}